\documentclass[12pt, reqno]{amsart}
\usepackage{amsmath, amsthm, amscd, amsfonts, amssymb, graphicx, color}
\usepackage[bookmarksnumbered, colorlinks, plainpages]{hyperref}
\hypersetup{colorlinks=true,linkcolor=red, anchorcolor=green, citecolor=cyan, urlcolor=red, filecolor=magenta, pdftoolbar=true}
\input{mathrsfs.sty}
\usepackage{xcolor}
\usepackage{xcolor}
\usepackage{ulem}
\newtheorem{theorem}{Theorem}[section]
\newtheorem{lemma}[theorem]{Lemma}
\newtheorem{prop}[theorem]{Proposition}
\newtheorem{cor}[theorem]{Corollary}
\theoremstyle{definition}

\newtheorem{example}[theorem]{Example}

\theoremstyle{remark}
\newtheorem{remark}[theorem]{\bf{Remark}}
\numberwithin{equation}{section}
\begin{document}
\title [Berezin number and  Berezin norm inequalities for operator matrices]{Berezin number and  Berezin norm inequalities for operator matrices}
	\author[P. Bhunia, A. Sen, S. Barik and K. Paul]{Pintu Bhunia, Anirban Sen, Somdatta Barik  and Kallol Paul}

\address[Bhunia]{Department of Mathematics, Indian Institute of Science, Bengaluru 560012, Karnataka, India}
\email{pintubhunia5206@gmail.com}

\address[Sen] {Department of Mathematics, Jadavpur University, Kolkata 700032, West Bengal, India}
\email{anirbansenfulia@gmail.com}

\address[Barik] {Department of Mathematics, Jadavpur University, Kolkata 700032, West Bengal, India}
\email{bariksomdatta97@gmail.com}

\address[Paul] {Department of Mathematics, Jadavpur University, Kolkata 700032, West Bengal, India}
\email{kalloldada@gmail.com}

	\thanks{Dr. Pintu Bhunia would like to thank SERB, Govt. of India for the financial support in the form of National Post Doctoral Fellowship (N-PDF, File No. PDF/2022/000325) under the mentorship of Prof. Apoorva Khare. Mr. Anirban Sen would like to thank CSIR, Govt. of India, for the financial
		support in the form of Senior Research Fellowship under the mentorship of
		Prof. Kallol Paul. Miss Somdatta Barik would like to thank UGC, Govt. of India for the financial support in the form of Junior Research Fellowship  under the mentorship of Prof. Kallol Paul. }



\begin{abstract}
	We establish new upper bounds for Berezin number and Berezin norm of operator matrices, which are refinements of  the existing bounds. Among other bounds, we prove that if $A=[A_{ij}]$ is an $n\times n$ operator matrix with $A_{ij}\in\mathbb{B}(\mathcal{H})$ for $i,j=1,2\dots n$, then 
	$\|A\|_{ber}	\leq \left\|\left[\|A_{ij}\|_{ber}\right]\right\|$ and $\textbf{ber}(A) \leq w([a_{ij}]),$ where $a_{ii}=\textbf{ber}(A_{ii}),$ $a_{ij}=\big\||A_{ij}|+|A^*_{ji}|\big\|^{\frac{1}{2}}_{ber} \big\||A_{ji}|+|A^*_{ij}|\big\|^{\frac{1}{2}}_{ber}$ if $i<j$ and $a_{ij}=0$ if $i>j$. Further, we give some examples for the Berezin number and Berezin norm estimation of operator matrices on the Hardy-Hilbert space.
\end{abstract}

\maketitle	
\section{Introduction}	
Let $\mathbb{H}$ be a complex Hilbert space endowed with an inner product $\langle \cdot,\cdot \rangle$ and $\mathbb{B}(\mathbb{H})$ be the collection of all bounded linear operators on $\mathbb{H}.$ An operator $A \in \mathbb{B}(\mathbb{H})$ is positive if $\langle Ax,x \rangle \geq 0$ for all $x \in \mathbb{H}.$ For $A \in \mathbb{B}(\mathbb{H})$, the absolute value of $A$ is the positive operator $|A|=(A^*A)^{1/2}.$ Recall that, the numerical range of $A$ is denoted by $W(A)$ and is defined as $W(A)=\left\{\langle Ax,x\rangle : x \in \mathbb{H},\|x\|=1\right\}.$ The famous Toeplitz-Hausdroff theorem states that the numerical range of an operator is a convex set. The usual operator norm and the numerical radius of $A$ are, respectively, defined by $\|A\|=\sup\left\{|\langle Ax,y\rangle| : x,y \in \mathbb{H},\|x\|=\|y\|=1\right\}$ and $w(A)=\sup\left\{|\lambda| : \lambda \in W(A)\right\}.$ It is well known that $w(\cdot)$ defines a norm on $\mathbb{B}(\mathbb{H})$ and satisfies the following inequality:
\begin{align}\label{E1}
	\frac{\|A\|}{2} \leq w(A) \leq \|A\|.
\end{align}
The inequality \eqref{E1} is sharp, in fact $w(A)=\frac{\|A\|}{2}$  if $A^2=0$ and $w(A)=\|A\|$ if $A$ is normal. For more on numerical radius inequalities and related results  we refer the readers to \cite{Bhunia_book,BP_LAA_21,BP_RM_21,BP_BSM_21,BP_AM_2021,CFK_NFAO_23,KZ_JCAM_23} and references therein.

Let $\Omega$ be a non empty set. A reproducing kernel Hilbert space (RKHS in short) $\mathcal{H}=\mathcal{H}(\Omega)$ is a Hilbert space of complex valued functions on the set $\Omega,$ where point evaluations are continuous, i.e., for each $\lambda \in \Omega$ the map $E_{\lambda} : \mathcal{H} \to \mathbb{C}$ defined by  $E_{\lambda}(f)=f(\lambda)$ is a bounded linear functional on $\mathcal{H}$ (see \cite{Paulsen_book}). By the Riesz representation theorem, for each $\lambda \in \Omega$ there exists a unique element $k_{\lambda} \in \mathcal{H}$ such that $E_{\lambda}(f)=\langle f,k_{\lambda} \rangle$ for all $f \in \mathcal{H}.$ The collection of functions  $\{k_\lambda :  \lambda \in \Omega \}$ is the set of all reproducing kernel of $\mathcal{H}$ and $\{\hat{k}_{\lambda}=k_\lambda/\|k_\lambda\| :  \lambda \in \Omega\}$ is the set of all normalized reproducing kernel of $\mathcal{H}$.  For  $A \in  \mathbb{B}(\mathcal{H}),$ the function $\widetilde{A}$ defined on $\Omega$ by $\widetilde{A}(\lambda)=\langle A\hat{k}_{\lambda},\hat{k}_{\lambda} \rangle$,  is called the Berezin symbol of $A$ (see \cite{BER1974,BER1972}). We recall that the Berezin Set, Berezin number and Berezin norm of $A$ are denoted by $\textbf{Ber}(A),\textbf{ber}(A)$ and $\|A\|_{ber},$ respectively, and defined as $\textbf{Ber}(A)=\left\{\widetilde{A}(\lambda) : \lambda \in \Omega\right\},$
$\textbf{ber}(A)=\sup\left\{|\widetilde{A}(\lambda)| : \lambda \in \Omega\right\}$
and $\|A\|_{ber}=\sup \left\{ \big|\langle A\hat{k}_{\lambda},\hat{k}_{\mu} \rangle\big| : \lambda, \mu \in \Omega\right\}$ (see \cite{BY_JIA_2020,KAR_JFA_2006}).  It is clear from the definitions that $\textbf{Ber}(A) \subseteq W(A),$ $\textbf{ber}(A)\leq w(A)$ and $\textbf{ber}(A) \leq \|A\|_{ber} \leq \|A\|.$ Although Berezin set is a subset of numerical range but unlike numerical range, the Berezin set of an operator is not convex, in general. For more about the characterization of convexity of Berezin range interested readers can follow the article \cite{CC_LAA_2022}. 
It should be mentioned that $\textbf{ber}(\cdot)$ and $\|\cdot\|_{ber}$ do not generally define a norm on $\mathbb{B}(\mathcal{H})$ but if the RKHS $\mathcal{H}$ has the ``Ber" property (i.e., for any two operators $A,B \in \mathbb{B}\left(\mathcal{H}\right)$,   $\widetilde{A}\left(
\lambda\right)=\widetilde{B}\left(\lambda\right)$ for all $\lambda \in \Omega$ implies $A=B$) then they define a norm on $\mathbb{B}(\mathcal{H}).$
The Berezin symbol of an operator provides important information about the operator and it has wide application in operator theory. It has been studied in details for Toeplitz and Hankel operators on Hardy and Bergman spaces. For further information about the Berezin symbol and its application, we refer the readers to see \cite{K2013, KAR_JFA_2006,KS_CVTA_2005,T_OM_2021}. Recall that the Hardy-Hilbert space of the unit disk  $\mathbb{D}= \{\lambda \in \mathbb{C} : |\lambda |<1 \}$ is denoted by ${H}^2(\mathbb{D}),$ and is defined as the Hilbert space of all square summable analytic functions on $\mathbb{D}$. It is well known that ${H}^2(\mathbb{D})$ is a RKHS and for $\lambda \in \mathbb{D}$ the corresponding reproducing kernel of ${H}^2(\mathbb{D})$ is given by $k_{\lambda}(z)=\sum_{n=0}^{\infty}\bar{\lambda}^n z^n$ (see \cite{Paulsen_book}).  
 Recently it is proved in \cite[Proposition 2.11]{BP_sen_CAOT} that the equality $\|A\|_{ber}=\textbf{ber}(A)$ holds if $A$ is positive. Observe that  the above equality may not be true for  selfadjoint operators. The Berezin number and Berezin norm inequalities have been studied by many mathematicians  \cite{BGKST_NFAO_2023,GA_CAOT_2021,GGO2016,MMM_JMAA_23,STC_RMJM_2022,sen_RMJM,YI_LAMA_2022,YG_NYJM_2017}.

Let $\left\{ \Omega_i : i=1,2,\ldots,n \right\}$ be a collection of nonempty sets and $\mathcal{H}_i=\mathcal{H}(\Omega_i)$ be a RKHS on $\Omega_i$ for each $i.$	Let us consider the direct sum $\mathcal{H}=\oplus^n_{i=1}\mathcal{H}_i.$ Then it is easy to observe that $\mathcal{H}$ is a RKHS on the nonempty set $\Omega_1\times \Omega_2\times \ldots \times \Omega_n.$ Every operator $A \in \mathbb{B}(\mathcal{H})$ has an $n \times n$ operator matrix representation $A=[A_{ij}]_{n \times n},$ where $A_{ij} \in \mathbb{B}(\mathcal{H}_j,\mathcal{H}_i).$ Here $\mathbb{B}(\mathcal{H}_j,\mathcal{H}_i)$ is the collection of all bounded linear operators from $\mathcal{H}_j$ to $\mathcal{H}_i.$ In recent years many mathematicians developed useful bounds of the numerical radius and operator norm of $n \times n$ operator matrices (see \cite{AK_LAA_2015, BK_LAMA_2009,Bhunia_arxiv_23,HD_IEOT_95}). Motivated by this idea Berezin number inequalities for operator matrices were also studied in \cite{Bak_CMJ_2018,SDR_ACTA_2021}.

In this paper, we obtain upper bounds for the Berezin number and Berezin norm of $n \times n$ operator matrices. As a special case for $2 \times 2$ operator matrices the bounds obtained here are  better than the  existing ones. We also develop several upper bounds for single and product operators. 

\section{main results}	

We begin this section with the following lemmas which will be used often.
\begin{lemma}\label{lm1}\cite[p. 20]{Sim_CUP_79}
	Let $A\in \mathbb{B}(\mathcal{H})$ be positive and let $x\in \mathcal{H}$ with $\|x\|=1.$ Then 
	\begin{eqnarray*}
		\langle Ax,x\rangle^r\leq \langle A^rx,x\rangle~~\,\, \text{for all}~~ r\geq 1.
	\end{eqnarray*}
\end{lemma}

\begin{lemma}\label{lm8}\cite[pp. 75-76]{hal_book_82}
	Let $A\in \mathbb{B}(\mathcal{H})$ and let $x, y\in \mathcal{H}$. Then 
	\[|\langle Ax,y\rangle|^2\leq \langle |A|x,x\rangle \langle |A^*|y,y\rangle.\]
\end{lemma}

\begin{lemma}\label{lm5}\cite[Th. 5]{Kitt_PRIMS_88}
	Let $A\in\mathbb{B}(\mathcal{H}).$ Let $f$ and $g$ be continuous functions on $[0,\infty)$, which satisfy the relation $f(\lambda) g(\lambda)=\lambda$ for all $\lambda\in[0,\infty).$ Then 
	\[|\langle Ax,y \rangle|\leq \|f(|A|)x\|\|g(|A^*|)y\|,\] for all $ x,y \in \mathcal{H}.$
\end{lemma}
\begin{lemma}\label{lmi}\cite[p. 44]{HJ_CUP_91}
  If $A = [a_{ij}]$ is an $n \times n $ complex matrix such that $ a_{ij}\geq 0$ for all $i, j = 1, 2, . . . , n,$ then
  $w(A)=w(\frac{A+A^*}{2})=r(\frac{A+A^*}{2}),$ where r(·) denotes the spectral radius.
\end{lemma}

\begin{lemma}\label{lm6}\cite[Cor. 2.5]{KDM_Filo_12}
	Let $x_{1},x_{2},e\in \mathcal{H}$ with $\|e\|=1$ and $\alpha \in \mathbb{C}\backslash\{0\}.$ Then
	\[|\langle x_{1},e\rangle \langle e,x_{2} \rangle|\leq \frac{1}{|\alpha|}\big(\max\{1,|\alpha -1|\}\|x_{1}\|\|x_{2}\|+|\langle x_{1},x_{2} \rangle|\big).\]
	
	In particular, for $\alpha=2,$  the above inequality becomes the Buzano's inequality  \cite{buza_74} 
	\[|\langle x_{1},e\rangle \langle e,x_{2}\rangle|\leq \frac{1}{2}(\|x_1\|\|x_2\|+|\langle x_1,x_2\rangle|).\]
\end{lemma}
\begin{lemma}\label{lm7}\cite[p. 1001]{Bak_CMJ_2018}
	Let $A\in \mathbb{B}(\mathcal{H}_{1}),B\in \mathbb{B}(\mathcal{H}_{2},\mathcal{H}_{1}),C\in\mathbb{B}(\mathcal{H}_{1},\mathcal{H}_{2})$ and $D\in\mathbb{B}(\mathcal{H}_{2}).$
	Then the following inequalities hold:
\begin{eqnarray*}
&(i)&~~\textbf{ber}\begin{pmatrix}
			\begin{bmatrix}
				A & 0\\
				0 & D
			\end{bmatrix} 
		\end{pmatrix}\leq \max\Big\{\textbf{ber}(A),\textbf{ber}(D)\Big\},\\
&(ii)&~~\textbf{ber}\begin{pmatrix}
			\begin{bmatrix}
				0 & B\\
				C & 0
			\end{bmatrix} 
		\end{pmatrix}\leq \frac{1}{2}\big(\|B\|+\|C\|\big).
\end{eqnarray*}

In particular, if $\mathcal{H}_1=\mathcal{H}_2$ then
			$\textbf{ber}\begin{pmatrix}
			\begin{bmatrix}
				0 & B\\
				B & 0
			\end{bmatrix} 
		\end{pmatrix}\leq \|B\|.$
\end{lemma}

Now, we are in a position to prove our first result which provides an upper bound of the Berezin number of $n \times n$ operator matrices.

\begin{theorem}\label{th4}
	Let $A=[A_{ij}]$ be an $n\times n$ operator matrix, where $A_{ij}\in\mathbb{B}(\mathcal{H})$ for all $i,j=1,2,...n.$ If $f,g:[0,\infty)\mapsto [0,\infty)$ are continuous functions satisfying $f(\lambda) g(\lambda)=\lambda,$ for all $\lambda\in[0,\infty),$ then
	\begin{eqnarray*}
	\textbf {ber}\begin{pmatrix}
		\begin{bmatrix}
		    A_{11} &A_{12}& \dots& A_{1n} \\
			A_{21} &A_{22}& \dots& A_{2n} \\
			\vdots &\vdots& \ddots&\vdots \\
			A_{n1} &A_{n2}& \dots &A_{nn}
	\end{bmatrix}\end{pmatrix}\leq
	w\begin{pmatrix}
		\begin{bmatrix}
			\textbf{ber}(A_{11})&a_{12}& \dots& a_{1n} \\
			0 &\textbf{ber}(A_{22})& \dots& a_{2n} \\
			\vdots &\vdots& \ddots&\vdots \\
			0 &0& \dots &\textbf{ber}(A_{nn})
	\end{bmatrix}\end{pmatrix}
\end{eqnarray*}
	where $a_{ij}=\big\|f^2(|A_{ij}|)+g^2(|A^*_{ji}|)\big \|^{\frac12}_{ber} \,\big\|f^2(|A_{ji}|)+g^2(|A^*_{ij}|)\big\|^{\frac12}_{ber} .$
\end{theorem}
\begin{proof}
	 For $(\lambda_{1}, . . . , \lambda_{n})\in \Omega^n =\Omega\times . . . \times  \Omega$ ($n$-copies), let $\hat{k}_{(\lambda_{1}, . . . , \lambda_{n})}$ be the corresponding normalized reproducing kernel of $\mathcal{H}\oplus \ldots \oplus \mathcal{H}$ ($n$-copies). Then
	\begin{eqnarray*}
		&&|\langle A\hat{k}_{(\lambda_{1}, . . . , \lambda_{n})}, \hat{k}_{(\lambda_{1}, . . . , \lambda_{n})}\rangle|\\
		&=& |\sum_{i,j=1}^{n}\langle A_{ij}k_{\lambda_{j}},k_{\lambda_{i}}\rangle |\\
		&\leq & \sum_{i,j=1}^{n}|\langle A_{ij}k_{\lambda_{j}},k_{\lambda_{i}}\rangle|\\
		&=& \sum_{i=1}^{n}|\langle A_{ii}k_{\lambda_{i}},k_{\lambda_{i}}\rangle|+\sum_{i,j=1 \atop i\neq j}^{n}|\langle A_{ij}k_{\lambda_{j}},k_{\lambda_{i}}\rangle|\\
		&=&\sum_{i=1}^{n}|\langle A_{ii}k_{\lambda_{i}},k_{\lambda_{i}}\rangle|+\sum_{i,j=1 \atop i<j}^{n}\Big(|\langle A_{ij}k_{\lambda_{j}},k_{\lambda_{i}}\rangle|+|\langle A_{ji}k_{\lambda_{i}},k_{\lambda_{j}}\rangle|\Big)\\
		&\leq & \sum_{i=1}^{n}|\langle A_{ii}k_{\lambda_{i}},k_{\lambda_{i}}\rangle|\\
		&&+\sum_{i,j=1 \atop i<j}^{n}\Big(\|f(|A_{ij}|)k_{\lambda_{j}}\|\|g(|A^*_{ij}|)k_{\lambda_{i}}\|+\|f(|A_{ji}|)k_{\lambda_{i}}\|\|g(|A^*_{ji}|)k_{\lambda_{j}}\|\Big)\Big(\mbox{by Lemma \ref{lm5}}\Big)\\
		&\leq &\sum_{i=1}^{n}|\langle A_{ii}k_{\lambda_{i}},k_{\lambda_{i}}\rangle|\\
		&&+\sum_{i,j=1 \atop i<j}^{n}\Big(\|f(|A_{ij}|)k_{\lambda_{j}}\|^2+\|g(|A^*_{ji}|)k_{\lambda_{j}}\|^2\Big)^{\frac{1}{2}} \Big(\|f(|A_{ji}|)k_{\lambda_{i}}\|^2+\|g(|A^*_{ij}|)k_{\lambda_{i}}\|^2\Big)^{\frac{1}{2}}\\
		&=& \sum_{i=1}^{n}|\langle A_{ii}k_{\lambda_{i}},k_{\lambda_{i}}\rangle|\\
		&&+\sum_{i,j=1 \atop i<j}^{n}\langle (f^2(|A_{ij}|)+g^2(|A^*_{ji}|))k_{\lambda_{j}},k_{\lambda_{j}}\rangle ^{\frac{1}{2}}
		\langle (f^2(|A_{ji}|)+g^2(|A^*_{ij}|))k_{\lambda_{i}},k_{\lambda_{i}}\rangle ^{\frac{1}{2}}\\
		&\leq& \sum_{i=1}^{n}\textbf{ber}(A_{ii})\|k_{\lambda_{i}}\|^2\\
		&&+\sum_{i,j=1 \atop i<j}^{n} \big\|f^2(|A_{ij}|)+g^2(|A^*_{ji}|)\big \|^{\frac12}_{ber} \big\|f^2(|A_{ji}|)+g^2(|A^*_{ij}|)\big\|^{\frac12}_{ber}  \|k_{\lambda_{i}}\| \|k_{\lambda_{j}}\|\\
		&=& \langle \hat{A}y,y \rangle, 		
	\end{eqnarray*}
where $y=\begin{bmatrix}
	\|k_{\lambda_{1}}\|\\
	\|k_{\lambda_{2}}\|\\
	\vdots\\
	\|k_{\lambda_{n}}\|
	\end{bmatrix} \in \mathbb{C}^n$ 
is an unit vector and $\hat{A}=[a_{ij}]$ is an $n\times n$ complex matrix, with
\[a_{ij}=\begin{cases}
	\textbf{ber}(A_{ii}) & \text{when } i=j,\\
	\big\|f^2(|A_{ij}|)+g^2(|A^*_{ji}|)\big \|^{\frac12}_{ber} \, \big\|f^2(|A_{ji}|)+g^2(|A^*_{ij}|)\big\|^{\frac12}_{ber}  & \text{when } i<j,\\
	0 & \text{otherwise}.
\end{cases}\]
Since $\|y\|=1,$ so \[|\langle A\hat{k}_{(\lambda_{1}, . . . , \lambda_{n})}, \hat{k}_{(\lambda_{1}, . . . , \lambda_{n})}\rangle|\leq w(\hat{A}).\]
Taking the supremum over all $(\lambda_{1}, . . . , \lambda_{n})\in \Omega^n$, we get
$\textbf{ber}(A)\leq w(\hat{A}).$
\end{proof}
By considering $f(\lambda)=\lambda ^t$ and $g(\lambda)=\lambda^{1-t} (0\leq t \leq 1)$ in Theorem \ref{th4}, we obtain the following corollary.
\begin{cor}\label{co1}
	If $ A_{ij}\in\mathbb{B}(\mathcal{H})$ for all $i,j=1,2,...n,$ then
	\[\textbf{ber}\begin{pmatrix}
	\begin{bmatrix}
		A_{11} &A_{12}& \dots& A_{1n} \\
		A_{21} &A_{22}& \dots& A_{2n} \\
		\vdots &\vdots& \ddots&\vdots \\
		A_{n1} &A_{n2}& \dots &A_{nn}
	\end{bmatrix}\end{pmatrix}\leq
 w\begin{pmatrix}
 	\begin{bmatrix}
 		\textbf{ber}(A_{11})&a_{12}& \dots& a_{1n} \\
    	0 &\textbf{ber}(A_{22})& \dots& a_{2n} \\
 		\vdots &\vdots& \ddots&\vdots \\
 		0 &0& \dots &\textbf{ber}(A_{nn})
 \end{bmatrix}\end{pmatrix}\]
where $a_{ij}=\||A_{ij}|^{2t}+|A^*_{ji}|^{2(1-t)}\|^{\frac{1}{2}}_{ber}  \, \||A_{ji}|^{2t}+|A^*_{ij}|^{2(1-t)}\|^{\frac{1}{2}}_{ber}$ and $ t\in [0,1].$
\end{cor}
The following corollary follows from Corollary \ref{co1} by considering $t=\frac{1}{2}.$
\begin{cor}\label{co2}
		If $ A_{ij}\in\mathbb{B}(\mathcal{H})$ for all $i,j=1,2,...n,$ then
		\[\textbf{ber}\begin{pmatrix}
			\begin{bmatrix}
				A_{11} &A_{12}& \dots& A_{1n} \\
				A_{21} &A_{22}& \dots& A_{2n} \\
				\vdots &\vdots& \ddots&\vdots \\
				A_{n1} &A_{n2}& \dots &A_{nn}
		\end{bmatrix}\end{pmatrix}\leq
		w\begin{pmatrix}
			\begin{bmatrix}
				\textbf{ber}(A_{11})&a_{12}& \dots& a_{1n} \\
				0 &\textbf{ber}(A_{22})& \dots& a_{2n} \\
				\vdots &\vdots& \ddots&\vdots \\
				0 &0& \dots &\textbf{ber}(A_{nn})
		\end{bmatrix}\end{pmatrix}\]
	where $a_{ij}=\big\||A_{ij}|+|A^*_{ji}|\big\|^{\frac{1}{2}}_{ber}  \, \big\||A_{ji}|+|A^*_{ij}|\big\|^{\frac{1}{2}}_{ber}.$
\end{cor}
Considering $n=2$ in Theorem \ref{th4}, we obtain the following upper bound of Berezin number for $\lowercase{2}\times \lowercase{2}$ operator matrices.
\begin{cor}\label{co5}
	Let  $A=\begin{bmatrix}
		A_{11} & 	A_{12}\\
		A_{21} & 	A_{22}
	\end{bmatrix}\in \mathbb{B}(\mathcal{H} \oplus \mathcal{H}).$ Then 
\begin{eqnarray*}
	\textbf{ber}(A)&\leq& \frac{1}{2} \Bigg(\textbf{ber}(A_{11})+\textbf{ber}(A_{22})\\
	&&+\sqrt{ \big(\textbf{ber}(A_{11})-\textbf{ber}(A_{22})\big)^2+\big\||A_{12}|+|A^*_{21}|\big\|_{ber} \big\||A_{21}|+|A^*_{12}|\big\|_{ber}}\Bigg).
\end{eqnarray*}
\end{cor}
\begin{proof}
	Taking $n=2$ in Theorem \ref{th4}, we get the inequality
	\begin{eqnarray*}
		\textbf{ber}\begin{pmatrix}
			\begin{bmatrix}
					A_{11} & 	A_{12}\\
				A_{21} & 	A_{22}
			\end{bmatrix}
		\end{pmatrix}&\leq& w \begin{pmatrix}
		\begin{bmatrix}
		\textbf{ber}(A_{11}) & 	a_{12}\\
			0 & 		\textbf{ber}(A_{22})
		\end{bmatrix}
	\end{pmatrix} \\
     &=& w \begin{pmatrix}
     	\begin{bmatrix}
     		\textbf{ber}(A_{11}) & 	\frac{1}{2}a_{12}\\
     		\frac{1}{2}a_{12} &  \textbf{ber}(A_{22})
     	\end{bmatrix}
     \end{pmatrix}     \Big(\mbox{by Lemma \ref{lmi}}\Big)  \\
    &=& r\begin{pmatrix}
    	\begin{bmatrix}
    		\textbf{ber}(A_{11}) & 	\frac{1}{2}a_{12}\\
    		\frac{1}{2}a_{12} &  \textbf{ber}(A_{22})
    	\end{bmatrix}
    \end{pmatrix} \\
&=& \frac{1}{2} \Bigg(\textbf{ber}(A_{11})+\textbf{ber}(A_{22})\\
&&+\sqrt{ \big(\textbf{ber}(A_{11})-\textbf{ber}(A_{22})\big)^2+ a_{12}^2}\Bigg),
	\end{eqnarray*} 
where $a_{12}=\big\||A_{12}|+|A^*_{21}|\big\|^{\frac{1}{2}}_{ber} \big\||A_{21}|+|A^*_{12}|\big\|^{\frac{1}{2}}_{ber}.$

\end{proof}

In particular, if  $A=\begin{bmatrix}
	0 & 	A_{12}\\
	A_{21} & 0
\end{bmatrix}$ then 
\[ \textbf{ber}\begin{pmatrix}
	\begin{bmatrix}
		0 & 	A_{12}\\
		A_{21} & 0
	\end{bmatrix}
\end{pmatrix}\leq\frac{1}{2} \big\||A_{12}|+|A^*_{21}|\big\|^{\frac{1}{2}}_{ber}  \, \big\||A_{21}|+|A^*_{12}|\big\|^{\frac{1}{2}}_{ber}.\]

In the following example we compute an upper bound of Berezin number for a $2 \times 2$ operator matrix by applying Corollary \ref{co5}.

\begin{example}
	Suppose  $\mathbb{M}=\begin{bmatrix}
		M_z & 	P_{\mathbb{C}}\\
		P_z & M_{z^2}
	\end{bmatrix} \in \mathbb{B}({H}^2(\mathbb{D})\oplus{H}^2(\mathbb{D}))$,  where $P_{\mathbb{C}}, P_z, M_z$ and $M_{z^2}$ are respectively defined as  $P_{\mathbb{C}}(f(z))=\langle f(z),1 \rangle,
	 P_z(f(z))=\langle f(z),z \rangle z, M_z(f(z))=zf(z)$ and $M_{z^2}(f(z))=z^2f(z)$ ($f \in {H}^2(\mathbb{D})$, $z\in \mathbb{D}$). Then a simple computation shows that $\textbf{ber}(M_z)=\textbf{ber}(M_{z^2})=1$ and $\||P_{\mathbb{C}}|+|P^*_z|\big\|_{ber} =\||P_z|+|P^*_{\mathbb{C}}|\|_{ber}=1.$ From Corollary \ref{co5}, it follows that
	 \begin{eqnarray*}
	 	\textbf{ber}(\mathbb{M})&\leq& \frac{1}{2} \Bigg(\textbf{ber}(M_z)+\textbf{ber}(M_{z^2})\\
	 	&&+\sqrt{ \big(\textbf{ber}(M_z)-\textbf{ber}(M_{z^2})\big)^2+\||P_{\mathbb{C}}|+|P^*_z|\big\|_{ber} \||P_z|+|P^*_{\mathbb{C}}|\|_{ber}}\Bigg)\\
	 	&=& 1.5.
	 \end{eqnarray*}
\end{example}

\begin{remark}
	(i) In \cite[Corollary 2.2]{Bak_CMJ_2018}, Bakherad obtained the following bound, namely,
	\begin{eqnarray}\label{eqn12}
		\textbf{ber}\begin{pmatrix}
		\begin{bmatrix}
			A_{11} & 	A_{12}\\
			A_{21} & 	A_{22}
		\end{bmatrix}
	\end{pmatrix}
	&\leq&\frac{1}{2} \Bigg(\textbf{ber}(A_{11})+\textbf{ber}(A_{22})\nonumber\\
	&+ & \sqrt{ \big(\textbf{ber}(A_{11})-\textbf{ber}(A_{22})\big)^2+(\|A_{12}\|+\|A_{21}\|)^2 }\Bigg).
	\end{eqnarray}
If we consider $\mathcal{H}=\mathbb{C}^2, A_{11}=A_{22}=\begin{bmatrix}
	1 & 0\\
	0 & 0
\end{bmatrix}$ and $A_{12}=A_{21}=\begin{bmatrix}
0 & 1\\
0 & 0
\end{bmatrix}$ then from the bound in Corollary \ref{co5} we get $\textbf{ber}\begin{pmatrix}
\begin{bmatrix}
	A_{11} & 	A_{12}\\
	A_{21} & 	A_{22}
\end{bmatrix}
\end{pmatrix}
\leq 1.5,$ whereas the inequality (\ref{eqn12}) gives  $\textbf{ber}\begin{pmatrix}
	\begin{bmatrix}
		A_{11} & 	A_{12}\\
		A_{21} & 	A_{22}
	\end{bmatrix}
\end{pmatrix}
\leq 2.$ Therefore, for this example the bound of Corollary \ref{co5} is better than the existing bound  \eqref{eqn12}.\\
(ii) The following upper bound  
\begin{eqnarray}\label{R1E2}
	\textbf{ber}\begin{pmatrix}
		\begin{bmatrix}
			A_{11} & 	A_{12}\\
			A_{21} & 	A_{22}
		\end{bmatrix}
	\end{pmatrix}
	&\leq & \frac{1}{2} \Bigg(\textbf{ber}(A_{11})+\textbf{ber}(A_{22})\nonumber\\
	&+& \sqrt{ \big(\textbf{ber}(A_{11})-\textbf{ber}(A_{22})\big)^2+4w^2\begin{pmatrix}
			\begin{bmatrix}
				0& 	A_{12}\\
				A_{21} & 0
			\end{bmatrix}
	\end{pmatrix} }\Bigg)
\end{eqnarray}
was obtained in \cite[Corollary 3.2]{SDR_ACTA_2021}. Consider $	A_{11}=	A_{22}=0,$ $	A_{12}=	\begin{bmatrix}
	1 & 1\\
	0 & 0
\end{bmatrix}$ 
and $	A_{21}=	\begin{bmatrix}
	1 & 0\\
	1 & 0
\end{bmatrix}.$ Then it is easy to observe that  Corollary \ref{co5} gives  $\textbf{ber}\begin{pmatrix}
\begin{bmatrix}
	A_{11} & 	A_{12}\\
	A_{21} & 	A_{22}
\end{bmatrix}
\end{pmatrix}
\leq 1,$ whereas \eqref{R1E2} gives $\textbf{ber}\begin{pmatrix}
	\begin{bmatrix}
		A_{11} & 	A_{12}\\
		A_{21} & 	A_{22}
	\end{bmatrix}
\end{pmatrix}
\leq \sqrt2.$ Hence the bound obtained in  Corollary \ref{co5} is better than that 
given in \eqref{R1E2}.
\end{remark}

Our next theorem yields an upper bound of the Berezin norm for $n \times n$ operator matrices.

	\begin{theorem}\label{th8}
	Let $A=[A_{ij}]$ be an $n\times n$ operator matrix with $A_{ij}\in\mathbb{B}(\mathcal{H}_j,\mathcal{H}_i),$ $1 \leq i,j\leq n.$ Then 
	\begin{eqnarray*}
		\|A\|_{ber}	\leq \left\|\left[\|A_{ij}\|_{ber}\right]\right\|,  1 \leq i,j\leq n.
	\end{eqnarray*}
\end{theorem}

\begin{proof}
	Let $\mathcal{H}=\oplus^n_{i=1}\mathcal{H}_i.$ For  $(\lambda_1,\lambda_2,\ldots,\lambda_n),(\mu_1,\mu_2,\ldots,\mu_n)\in \Omega_1\times \Omega_2\times \ldots \times \Omega_n,$ let 
	$\hat{k}_{(\lambda_1,\lambda_2,\ldots,\lambda_n)}=
	\begin{bmatrix}
		{k}_{\lambda_1}\\
		\vdots\\
		{k}_{\lambda_n}
	\end{bmatrix}$ and $\hat{k}_{(\mu_1,\mu_2,\ldots,\mu_n)}=
	\begin{bmatrix}
		{k}_{\mu_1}\\
		\vdots\\
		{k}_{\mu_n}
	\end{bmatrix}$
	be the corresponding normalized reproducing kernels of $\mathcal{H}.$ Then
	\begin{eqnarray*}
		|\langle A \hat{k}_{(\lambda_1,\lambda_2,\ldots,\lambda_n)},\hat{k}_{(\mu_1,\mu_2,\ldots,\mu_n)} \rangle|
		&& =\left| \sum_{i,j=1}^{n}\langle A_{ij}{k}_{\lambda_j},{k}_{\mu_i} \rangle \right|\\
		&& \leq \sum_{i,j=1}^{n}\left|\langle A_{ij}{k}_{\lambda_j},{k}_{\mu_i} \rangle \right|\\
		&&\leq\sum_{i,j=1}^{n}\|A_{ij}\|_{ber}\|{k}_{\lambda_j}\|\|{k}_{\mu_i}\|=\langle [\|A_{ij}\|_{ber}]x,y \rangle,
	\end{eqnarray*}
	where $x=\begin{bmatrix}
		\|{k}_{\lambda_1}\|\\
		\vdots\\
		\|{k}_{\lambda_n}\|\\
	\end{bmatrix}$
	and $y=\begin{bmatrix}
		\|{k}_{\mu_1}\|\\
		\vdots\\
		\|{k}_{\mu_n}\|\\
	\end{bmatrix}$.
	Since $\|x\|=\|y\|=1$,  we have 
	$$|\langle A \hat{k}_{(\lambda_1,\lambda_2,\ldots,\lambda_n)},\hat{k}_{(\mu_1,\mu_2,\ldots,\mu_n)} \rangle|\leq \|[\|A_{ij}\|_{ber}]\|.$$
	Therefore, taking the supremum over all $(\lambda_1,\lambda_2,\ldots,\lambda_n),(\mu_1,\mu_2,\ldots,\mu_n)\in \Omega_1\times \Omega_2\times \ldots \times \Omega_n,$ we get the desired result.
\end{proof}

Now, we give a computational example for an upper bound of the Berezin norm for a $2 \times 2$ operator matrix by applying Theorem \ref{th8}.
\begin{example}
	Let $\begin{bmatrix}
		P_{\mathbb{C}} &  P_z\\
		P_{z^2} &  P_{z^3}
	\end{bmatrix}\in \mathbb{B}({H}^2(\mathbb{D})\oplus {H}^2(\mathbb{D})),$ where $P_{\mathbb{C}}(f(z))=\langle f(z),1 \rangle$ and $P_{z^i}(f(z))=\langle f(z),z^i \rangle z^i$ $(f \in {H}^2(\mathbb{D})$, $z\in \mathbb{D})$ for $i=1,2,3.$ A short computation shows that $\|P_{\mathbb{C}}\|_{ber}=1,$ $\|P_{z}\|_{ber}=1/4,$ $\|P_{z^2}\|_{ber}=4/27$ and $\|P_{z^3}\|_{ber}=27/256.$ From Theorem \ref{th8}, it follows that $\left\|\begin{bmatrix}
		P_{\mathbb{C}} &  P_z\\
		P_{z^2} &  P_{z^3}
	\end{bmatrix}\right\|_{ber} \leq \left\|\begin{bmatrix}
		1 &  1/4\\
		4/27 &  27/256
	\end{bmatrix}\right\| \approx 1.045.$
\end{example}

The following inequalities concerning $\lowercase{2}\times \lowercase{2}$ operator matrices follows immediately from Theorem \ref{th8}.
\begin{cor}\label{c28}
	Let $A\in \mathbb{B}(\mathcal{H}_{1}), B\in \mathbb{B}(\mathcal{H}_{2},\mathcal{H}_{1}), C\in\mathbb{B}(\mathcal{H}_{1},\mathcal{H}_{2})$ and $ D\in\mathbb{B}(\mathcal{H}_{2}).$ Then
	\begin{eqnarray*}
	(i) ~~\left\|	\begin{bmatrix}
			A & 0\\
			0 & D
		\end{bmatrix}\right\|_{ber} \leq  \max \left\{\|A\|_{ber}, \|D\|_{ber}\right\},\\
	(ii) ~~\left\|	\begin{bmatrix}
			0 & B\\
			C & 0
		\end{bmatrix}\right\|_{ber} \leq  \max \left\{\|B\|_{ber}, \|C\|_{ber}\right\}. 
	\end{eqnarray*}
\end{cor}


\begin{example}
	Let  $\mathbb{P}=\begin{bmatrix}
		P_{\mathbb{C}} &  0\\
		0 &  P_z
	\end{bmatrix}\in \mathbb{B}({H}^2(\mathbb{D})\oplus {H}^2(\mathbb{D})),$  where $P_{\mathbb{C}}$ and $P_z$ on ${H}^2(\mathbb{D}),$ respectively defined as  $P_{\mathbb{C}}(f(z))=\langle f(z),1 \rangle$ and $P_z(f(z))=\langle f(z),z \rangle z$ ($f \in {H}^2(\mathbb{D})$, $z\in \mathbb{D}$). Then by simple computation it is easy to observe that $\|\mathbb{P}\|_{ber} =0.536, \left\|P_{\mathbb{C}}\right\|_{ber}=1$ and $\left\|P_{z}\right\|_{ber}=1/4.$ Therefore for this example we have $\left\|\begin{bmatrix}
		P_{\mathbb{C}} &  0\\
		0 &  P_z
	\end{bmatrix}\right\|_{ber} =0.536 <1= \max\{\left\|P_{\mathbb{C}}\right\|_{ber},\left\|P_{z}\right\|_{ber}\}.$ 
\end{example}

\begin{remark}
 By using the fact $\textbf{ber}\begin{pmatrix}
		\begin{bmatrix}
			0 &  B\\
			C &  0
		\end{bmatrix}
	\end{pmatrix}\leq\left\|	\begin{bmatrix}
	0 & B\\
	C & 0
	\end{bmatrix}\right\|_{ber}, $ from Corollary \ref{c28} (ii) we have
	 \begin{eqnarray}\label{eqn14}
	 	\textbf{ber}\begin{pmatrix}
		\begin{bmatrix}
			0 & B\\
			C & 0
		\end{bmatrix} 
	\end{pmatrix}\leq  \max \left\{\|B\|_{ber}, \|C\|_{ber}\right\}.
\end{eqnarray}
Now we give an example to show that the bounds in (\ref{eqn14}) and Lemma \ref{lm7} (ii) are not comparable, in general.
If we consider $\mathcal{H}_{1}=\mathcal{H}_{2}=\mathbb{C}^2,B= \begin{bmatrix}
	1 & 1\\
	0 & 0
\end{bmatrix}$ and $
C= \begin{bmatrix}
1 & 0\\
1 & 0
\end{bmatrix}$
then from (\ref{eqn14}), we get $\textbf{ber}\begin{pmatrix}
	\begin{bmatrix}
		0 & B\\
		C & 0
	\end{bmatrix} 
\end{pmatrix}\leq 1,$ whereas from Lemma \ref{lm7} (ii), we get  $\textbf{ber}\begin{pmatrix}
\begin{bmatrix}
	0 & B\\
	C & 0
\end{bmatrix} 
\end{pmatrix}\leq \sqrt{2}.$
Again, if we consider $B=\begin{bmatrix}
	1 & 0\\
	0 & 0
\end{bmatrix},
C=\begin{bmatrix}
0 & 0\\
0 & 2
\end{bmatrix}$
then the inequality (\ref{eqn14}) gives $\textbf{ber}\begin{pmatrix}
	\begin{bmatrix}
		0 & B\\
		C & 0
	\end{bmatrix} 
\end{pmatrix}\leq 2,$
whereas Lemma \ref{lm7} (ii) gives $\textbf{ber}\begin{pmatrix}
	\begin{bmatrix}
		0 & B\\
		C & 0
	\end{bmatrix} 
\end{pmatrix}\leq 1.5.$
\end{remark}

Next result gives an estimation for the Berezin number of  $2 \times 2$ off diagonal operator matrices.

\begin{theorem}\label{th5}
Let $  A\in \mathbb{B}(\mathcal{H}_{2},\mathcal{H}_{1}),  B\in\mathbb{B}(\mathcal{H}_{1},\mathcal{H}_{2})$ and $\alpha \in \mathbb{C}\backslash\{0\}.$ Then
 \begin{eqnarray*}
	\textbf{ber}^2\begin{pmatrix}
		\begin{bmatrix}
			0 & A\\
			B & 0
		\end{bmatrix} 
	\end{pmatrix}
  &\leq& \frac{1}{|\alpha|}\max\Big\{\textbf {ber}(AB),\textbf {ber}(BA)\Big\}\\
  &+& \frac{\max\{1,|\alpha-1|\}}{2|\alpha|} \max\Big\{\|AA^*+B^*B\|_{ber}, \|BB^*+A^*A\|_{ber}\Big\}.
\end{eqnarray*}
\end{theorem}
\begin{proof}
Let $T=\begin{bmatrix}
	0 & A\\
	B & 0
\end{bmatrix}.$
For $(\lambda_1,\lambda_2)\in \Omega_1\times \Omega_2,$ let  $\hat{k}_{(\lambda_1,\lambda_2)}$ be a normalized reproducing kernel of $\mathcal{H}_{1} \oplus \mathcal{H}_{2}.$ Then
\begin{eqnarray*}
    && |\langle T\hat{k}_{(\lambda_1,\lambda_2)},\hat{k}_{(\lambda_1,\lambda_2)}\rangle|^2\\
    &=&  |\langle T\hat{k}_{(\lambda_1,\lambda_2)},\hat{k}_{(\lambda_1,\lambda_2)}\rangle\langle \hat{k}_{(\lambda_1,\lambda_2)},T^*\hat{k}_{(\lambda_1,\lambda_2)}\rangle|\\
	&\leq& \frac{1}{|\alpha|}\Big(|\langle T\hat{k}_{(\lambda_1,\lambda_2)},T^*\hat{k}_{(\lambda_1,\lambda_2)}\rangle|+\max\{1,|\alpha-1|\} \|T\hat{k}_{(\lambda_1,\lambda_2)}\| \|T^*\hat{k}_{(\lambda_1,\lambda_2)}\|\Big)\\
	&&\,\,\,\,\,\,\,\,\,\,\,\,\,\,\,\,\,\,\,\,\,\,\,\,\,\,\,\,\,\, \,\,\,\,\,\,\,\,\,\,\,\,\,\,\,\,\,\,\,\,\,\,\,\,\,\,\,\,\,\,\,\,\,\, \,\,\,\,\,\,\,\,\,\,\,\,\,\,\,\,\,\,\,\,\,\,\,\,\,\,\,\,\,\,\,\,\,\, \,\,\,\,\,\,\,\,\,\,\,\,\,\,\,\,\,\,\,\,\,\,\,\,\,\,\,\,\,\,\,\,\,\, \,\,\,\, \Big(\mbox{by Lemma \ref{lm6}}\Big)\\
	&\leq&  \frac{1}{|\alpha|}\Bigg(|\langle T^2\hat{k}_{(\lambda_1,\lambda_2)},\hat{k}_{(\lambda_1,\lambda_2)}\rangle|+\frac{\max\{1,|\alpha-1|\}}{2} \left(\|T\hat{k}_{(\lambda_1,\lambda_2)}\|^2+ \|T^*\hat{k}_{(\lambda_1,\lambda_2)}\|^2\right)\Bigg)\\
	&=&\frac{1}{|\alpha|}\Bigg(|\langle T^2\hat{k}_{(\lambda_1,\lambda_2)},\hat{k}_{(\lambda_1,\lambda_2)}\rangle|+\frac{\max\{1,|\alpha-1|\}}{2} \langle (T^*T+TT^*)\hat{k}_{(\lambda_1,\lambda_2)},\hat{k}_{(\lambda_1,\lambda_2)}\rangle\Bigg)\\
	&\leq& \frac{1}{|\alpha|}\Bigg(\textbf{ber}\left(\begin{bmatrix}
		AB & 0\\
		0 & BA
	\end{bmatrix}\right)+\frac{\max\{1,|\alpha-1|\}}{2}\textbf{ber}\left(\begin{bmatrix}
	AA^*+B^*B & 0\\
	0 & A^*A+BB^*
	\end{bmatrix}\right)\Bigg)\\
	&\leq& \frac{1}{|\alpha|}\max\Big\{\textbf {ber}(AB),\textbf {ber}(BA)\Big\}\\
	&~~&+ \frac{\max\{1,|\alpha-1|\}}{2|\alpha|} \max\Big\{\|AA^*+B^*B\|_{ber}, \|BB^*+A^*A\|_{ber}\Big\} 	\,\,\,\,\Big(\mbox{by Lemma \ref{lm7}}(\text{i})\Big).
\end{eqnarray*}
 Taking the supremum over all $(\lambda_1,\lambda_2)\in \Omega_1\times \Omega_2,$ we get the desired result.
\end{proof}

The following corollary is obvious from  Theorem \ref{th5}.
\begin{cor}\label{co3}
If $\mathcal{H}_{1}=\mathcal{H}_{2}$ then from Theorem \ref{th5}, we have
\begin{eqnarray}\label{inq2}
\textbf{ber}^2\begin{pmatrix}
	\begin{bmatrix}
		0 & A\\
		A & 0
	\end{bmatrix}
		\end{pmatrix} \leq\frac{1}{2|\alpha|}  \max\{1,|\alpha-1|\}\|AA^*+A^*A\|_{ber}+\frac{1}{|\alpha|}\textbf{ber}(A^2).
\end{eqnarray}
 Considering $\alpha=n>1(n \in \mathbb{N})$ in (\ref{inq2}), we get
\begin{eqnarray*}\textbf{ber}^2\begin{pmatrix}
	\begin{bmatrix}
		0 & A\\
		A & 0
	\end{bmatrix}
\end{pmatrix} \leq \frac{n-1}{2n}\|AA^*+A^*A\|_{ber}+\frac{1}{n}\textbf{ber}(A^2).
\end{eqnarray*}
Now, taking  $n\rightarrow \infty$, we have 
\begin{eqnarray}\label{inq3}
	\textbf{ber}^2\begin{pmatrix}
		\begin{bmatrix}
			0 & A\\
			A & 0
		\end{bmatrix}
	\end{pmatrix} \leq\frac{1}{2}\|AA^*+A^*A\|_{ber}\leq \|A\|^2.\end{eqnarray}
	Therefore, the bound (\ref{inq3}) refines the existing bound in Lemma \ref{lm7} (ii).
\end{cor}

In particular, choosing $\alpha=n>1(n\in\mathbb{N})$ and taking $n\rightarrow \infty$ in Theorem \ref{th5}, we get the following corollary.
\begin{cor}\label{co4}
	Let $  A\in \mathbb{B}(\mathcal{H}_{2},\mathcal{H}_{1})$ and $ B\in\mathbb{B}(\mathcal{H}_{1},\mathcal{H}_{2}).$ Then
 \begin{eqnarray*}
	\textbf{ber}^2\begin{pmatrix}
		\begin{bmatrix}
			0 & A\\
			B & 0
		\end{bmatrix}
	\end{pmatrix}&\leq&\frac{1}{2} \max\Big\{\|AA^*+B^*B\|_{ber}, \|BB^*+A^*A\|_{ber}\Big\}. 
\end{eqnarray*}
\end{cor}

Next we develop the following upper bound for  $2 \times 2$ operator matrices.

\begin{theorem}\label{th6}
	Let $A\in \mathbb{B}(\mathcal{H}_{1}), B\in \mathbb{B}(\mathcal{H}_{2},\mathcal{H}_{1}), C\in\mathbb{B}(\mathcal{H}_{1},\mathcal{H}_{2}), D\in\mathbb{B}(\mathcal{H}_{2})$ and $\alpha \in \mathbb{C}\backslash\{0\}.$ Then
	\begin{eqnarray*}
		\textbf{ber}^2\begin{pmatrix}
			\begin{bmatrix}
				A & B\\
				C & D
			\end{bmatrix}
		\end{pmatrix}&\leq&\max\Big\{\textbf{ber}^2(A),\textbf{ber}^2(D)\Big\}+\max\Big\{\|B\|^{2}_{ber},\|C\|^{2}_{ber}\Big\}\\
		&&+\frac{\max\{1,|\alpha-1|\}}{|\alpha|} \max\Big\{\|A^*A+BB^*\|_{ber},\|CC^*+D^*D\|_{ber}\Big\}\\
		&&+\frac{2}{|\alpha|}\max\Big\{\|BD\|_{ber}, \|CA\|_{ber}\Big\}.
	\end{eqnarray*}
\end{theorem}
\begin{proof}
   Let $T=\begin{bmatrix}
   	A & 0 \\
   	0 & D
   \end{bmatrix} $ and $S=\begin{bmatrix}
   0 & B \\
   C & 0
   \end{bmatrix} .$  For $(\lambda_1,\lambda_2)\in \Omega_1\times \Omega_2,$ let $\hat{k}_{(\lambda_1,\lambda_2)}$ be a normalized reproducing kernel of $\mathcal{H}_{1} \oplus \mathcal{H}_{2}.$ Then
   \begin{eqnarray*}
   		&&\Bigg|\Bigg\langle\begin{bmatrix}
   		A & B\\
   		C & D
   	\end{bmatrix}\hat{k}_{(\lambda_1,\lambda_2)},\hat{k}_{(\lambda_1,\lambda_2)}\Bigg\rangle\Bigg|^2\\
   	&=&|\langle T\hat{k}_{(\lambda_1,\lambda_2)},\hat{k}_{(\lambda_1,\lambda_2)}\rangle+\langle S\hat{k}_{(\lambda_1,\lambda_2)},\hat{k}_{(\lambda_1,\lambda_2)}\rangle|^2\\
   &\leq& \big(|\langle T\hat{k}_{(\lambda_1,\lambda_2)},\hat{k}_{(\lambda_1,\lambda_2)}\rangle|+|\langle S\hat{k}_{(\lambda_1,\lambda_2)},\hat{k}_{(\lambda_1,\lambda_2)}\rangle|\big)^2\\
   &=&|\langle T\hat{k}_{(\lambda_1,\lambda_2)},\hat{k}_{(\lambda_1,\lambda_2)}\rangle|^2+|\langle S\hat{k}_{(\lambda_1,\lambda_2)},\hat{k}_{(\lambda_1,\lambda_2)}\rangle|^2\\
   &&+2|\langle T\hat{k}_{(\lambda_1,\lambda_2)},\hat{k}_{(\lambda_1,\lambda_2)}\rangle ||\langle S\hat{k}_{(\lambda_1,\lambda_2)},\hat{k}_{(\lambda_1,\lambda_2)}\rangle|\\
   &\leq& \textbf{ber}^2(T)+\textbf{ber}^2(S)+2|\langle T\hat{k}_{(\lambda_1,\lambda_2)},\hat{k}_{(\lambda_1,\lambda_2)}\rangle \langle \hat{k}_{(\lambda_1,\lambda_2)},S^*\hat{k}_{(\lambda_1,\lambda_2)}\rangle|\\
    &\leq&\textbf{ber}^2(T)+\textbf{ber}^2(S)+\frac{2}{|\alpha|} \max\{1,|\alpha-1|\}\|T\hat{k}_{(\lambda_1,\lambda_2)}\|\|S^*\hat{k}_{(\lambda_1,\lambda_2)}\|\\
    &&+\frac{2}{|\alpha|}|\langle T\hat{k}_{(\lambda_1,\lambda_2)},S^*\hat{k}_{(\lambda_1,\lambda_2)}\rangle| \,\,\,\,\,\,\,\,\, \Big(\mbox{by Lemma \ref{lm6}}\Big)\\
    &\leq& \textbf{ber}^2(T)+\textbf{ber}^2(S)+\frac{1}{|\alpha|} \max\{1,|\alpha-1|\}\big(\|T\hat{k}_{(\lambda_1,\lambda_2)}\|^2+\|S^*\hat{k}_{(\lambda_1,\lambda_2)}\|^2\big)\\
    &&+\frac{2}{|\alpha|}|\langle ST\hat{k}_{(\lambda_1,\lambda_2)},\hat{k}_{(\lambda_1,\lambda_2)}\rangle|\\
    &=&  \textbf{ber}^2(T)+\textbf{ber}^2(S)+\frac{1}{|\alpha|} \max\{1,|\alpha-1|\}\langle (T^*T+SS^*)\hat{k}_{(\lambda_1,\lambda_2)},\hat{k}_{(\lambda_1,\lambda_2)}\rangle\\
    &&+\frac{2}{|\alpha|}|\langle ST\hat{k}_{(\lambda_1,\lambda_2)},\hat{k}_{(\lambda_1,\lambda_2)}\rangle|\\
    &\leq& \textbf{ber}^2\left(\begin{bmatrix}
    	A & 0 \\
    	0 & D
    \end{bmatrix}\right)+\textbf{ber}^2\left(\begin{bmatrix}
    0 & B \\
    C & 0
    \end{bmatrix}\right)+\frac{\max\{1,|\alpha-1|\}}{|\alpha|} \\
    &&\textbf{ber}\left(\begin{bmatrix}
    A^*A+BB^* & 0 \\
    0 & CC^*+D^*D
    \end{bmatrix}\right)
    +\frac{2}{|\alpha|}\textbf{ber}\left(\begin{bmatrix}
    		0 & BD \\
    		CA & 0
    \end{bmatrix}\right)\\
    &\leq& \max\Big\{\textbf{ber}^2(A),\textbf{ber}^2(D)\Big\}+\max\Big\{\|B\|^{2}_{ber},\|C\|^{2}_{ber}\Big\}\\
    &&+\frac{\max\{1,|\alpha-1|\}}{|\alpha|} \max\Big\{\|A^*A+BB^*\|_{ber},\|CC^*+D^*D\|_{ber}\Big\}\\
    &&+\frac{2}{|\alpha|}\max\Big\{\|BD\|_{ber}, \|CA\|_{ber}\Big\}\,\,\,\,\Big(\mbox{by Lemma \ref{lm7} (i) and Corollary \ref{c28} (ii)}\Big).
   \end{eqnarray*}
    Taking the supremum over all $(\lambda_1,\lambda_2)\in \Omega_1\times \Omega_2,$ we get the desired result.
\end{proof}
For $\alpha=n>1 (n\in\mathbb{N})$ and taking $n\rightarrow\infty$  in Theorem \ref{th6}, we obtain the following bound.
\begin{cor}\label{co6}
	Let $A\in \mathbb{B}(\mathcal{H}_{1}), B\in \mathbb{B}(\mathcal{H}_{2},\mathcal{H}_{1}), C\in\mathbb{B}(\mathcal{H}_{1},\mathcal{H}_{2})$ and $ D\in\mathbb{B}(\mathcal{H}_{2}).$ Then
	\begin{eqnarray*}
	\textbf{ber}^2\begin{pmatrix}
		\begin{bmatrix}
			A & B\\
			C & D
		\end{bmatrix}
	\end{pmatrix}&\leq& \max\Big\{\textbf{ber}^2(A),\textbf{ber}^2(D)\Big\}+\max\Big\{\|B\|^{2}_{ber},\|C\|^{2}_{ber}\Big\}\\
	&&+\max\Big\{\|A^*A+BB^*\|_{ber},\|CC^*+D^*D\|_{ber}\Big\}.
	\end{eqnarray*}
\end{cor}
\begin{remark}
	If $\mathcal{H}_{1}=\mathcal{H}_{2}$ then from Corollary \ref{co6} we have
	\begin{eqnarray}\label{inq5}
			\textbf{ber}\begin{pmatrix}
			\begin{bmatrix}
				A & B\\
				B & A
			\end{bmatrix}
		\end{pmatrix}\leq \sqrt{\|B\|^{2}_{ber}+\textbf{ber}^2(A)+\|A^*A+BB^*\|_{ber}}.
	\end{eqnarray}
In \cite[Corollary 3.5]{hlb}, Hajmohamadi et al. obtained the following upper bound, namely,
	\begin{eqnarray}\label{inq6}
			\textbf{ber}\begin{pmatrix}
			\begin{bmatrix}
				A & B\\
				B & A
			\end{bmatrix}
		\end{pmatrix}\leq \frac{1}{2}\big(\textbf{ber}(|A|+|A^*|)+\textbf{ber}(|B|+|B^*|)\big).
	\end{eqnarray}
If we consider $\mathcal{H}_1=\mathcal{H}_2=\mathbb{C}^2, A=\begin{bmatrix}
	0 & 1\\
	0 & 0
\end{bmatrix}$ and $B=\begin{bmatrix}
1 & 0\\
0 & 0
\end{bmatrix}$
then  (\ref{inq5}) gives 
$\textbf{ber}\begin{pmatrix}
	\begin{bmatrix}
		A & B\\
		B & A
	\end{bmatrix}\end{pmatrix}\leq 1.41$ whereas (\ref{inq6}) gives $\textbf{ber}\begin{pmatrix}
\begin{bmatrix}
A & B\\
B & A
\end{bmatrix}\end{pmatrix}\leq 1.5.$ Hence for this example the bound of (\ref{inq5}) is finer than the existing bound (\ref{inq6}).
\end{remark}

Our next result reads as follows.

\begin{theorem}\label{th7}
Let	$A, B, C, D\in \mathbb{B}(\mathcal{H})$ and $\alpha \in \mathbb{C}\backslash\{0\}.$  Then
	\begin{eqnarray*}
&~~&	\left\|
		\begin{bmatrix}
			A & B\\
			C & D
		\end{bmatrix}
	\right\|_{ber}^2\\
	&\leq& \max\Big\{\textbf{ber}(|A|+i|C|),\textbf{ber}(|D|+i|B|)\Big\} 
\max\Big\{\textbf{ber}(|A^*|+i|B^*|),\textbf{ber}(|D^*|+i|C^*|)\Big\}\\
&~~&+\frac 12 \max\Big\{ \||A|^2+|C|^2\|_{ber},\||B|^2+|D|^2\|_{ber}\Big\}
+\max\Big\{\|C^*D\|_{ber}, \|B^*A\|_{ber}\Big\}.
\end{eqnarray*}
\end{theorem}
\begin{proof}
	Let $M=\begin{bmatrix}
		A & 0\\
		0 & D
	\end{bmatrix},$
 $N=\begin{bmatrix}
	0 & B\\
	C & 0
\end{bmatrix},$
$P=\begin{bmatrix}
	|A| & 0\\
	0 & |D|
\end{bmatrix},$
$Q=\begin{bmatrix}
	|A^*| & 0\\
	0 & |D^*|
\end{bmatrix},$
$R=\begin{bmatrix}
	|C| & 0\\
	0 & |B|
\end{bmatrix},$
$S=\begin{bmatrix}
	|B^*| & 0\\
	0 & |C^*|
\end{bmatrix}$ and
$T=\begin{bmatrix}
	0 & C^*D\\
	B^*A & 0
\end{bmatrix}.$
For $(\lambda_1,\lambda_2),(\mu_1,\mu_2)\in \Omega^2,$ let	$\hat{k}_{(\lambda_1,\lambda_2)}
$ and $\hat{k}_{(\mu_1,\mu_2)}
$ be two normalized reproducing kernel of $\mathcal{H} \oplus \mathcal{H}.$ 
 Therefore, $P=|M|,Q=|M^*|,R=|N|,S=|N^*|$ and $P^2+R^2=\begin{bmatrix}
 	|A|^2+|C|^2 & 0\\
 	0 & |D|^2+|B|^2
 \end{bmatrix},$
$P+iR=\begin{bmatrix}
	|A|+i|C|& 0\\
	0 & |D|+i|B|
\end{bmatrix}.$
 Then
 \begin{eqnarray*}
 	&&\Bigg|\Bigg\langle\begin{bmatrix}
 		A & B\\
 		C & D
 	\end{bmatrix}\hat{k}_{(\lambda_1,\lambda_2)},\hat{k}_{(\mu_1,\mu_2)}\Bigg\rangle\Bigg|^2\\
 &=& \left|\left\langle M\hat{k}_{(\lambda_1,\lambda_2)},\hat{k}_{(\mu_1,\mu_2)}\right\rangle+\left\langle N\hat{k}_{(\lambda_1,\lambda_2)},\hat{k}_{(\mu_1,\mu_2)}\right\rangle\right|^2\\
 &\leq&\Big[\left|\left\langle M\hat{k}_{(\lambda_1,\lambda_2)},\hat{k}_{(\mu_1,\mu_2)}\right\rangle\right|+\left|\left\langle N\hat{k}_{(\lambda_1,\lambda_2)},\hat{k}_{(\mu_1,\mu_2)}\right\rangle\right|\Big]^2\\
 &=&\left|\left\langle M\hat{k}_{(\lambda_1,\lambda_2)},\hat{k}_{(\mu_1,\mu_2)}\right\rangle\right|^2+\left|\left\langle N\hat{k}_{(\lambda_1,\lambda_2)},\hat{k}_{(\mu_1,\mu_2)}\right\rangle\right|^2\\
 &&+2\left|\left\langle M\hat{k}_{(\lambda_1,\lambda_2)},\hat{k}_{(\mu_1,\mu_2)}\right\rangle\right| \left|\left\langle \hat{k}_{(\mu_1,\mu_2)},N\hat{k}_{(\lambda_1,\lambda_2)}\right\rangle\right|\\
 &\leq& \left\langle |M|\hat{k}_{(\lambda_1,\lambda_2)},\hat{k}_{(\lambda_1,\lambda_2)}\right\rangle \left\langle |M^*|\hat{k}_{(\mu_1,\mu_2)},\hat{k}_{(\mu_1,\mu_2)}\right\rangle\\
 &&+\left\langle |N|\hat{k}_{(\lambda_1,\lambda_2)},\hat{k}_{(\lambda_1,\lambda_2)}\right\rangle \left\langle |N^*|\hat{k}_{(\mu_1,\mu_2)},\hat{k}_{(\mu_1,\mu_2)}\right\rangle
 + \left\|M\hat{k}_{(\lambda_1,\lambda_2)}\right\|\left\|N\hat{k}_{(\lambda_1,\lambda_2)}\right\|\\
 &&+\left|\left\langle M\hat{k}_{(\lambda_1,\lambda_2)},N\hat{k}_{(\lambda_1,\lambda_2)}\right\rangle\right|    \,\,\,\,\,\,\,\,\,\,\,\,     \Big(\mbox{by Lemma \ref{lm6} and Lemma \ref{lm8}}\Big)\\ 
 &=& \left\langle P\hat{k}_{(\lambda_1,\lambda_2)},\hat{k}_{(\lambda_1,\lambda_2)}\right\rangle \left\langle Q\hat{k}_{(\mu_1,\mu_2)},\hat{k}_{(\mu_1,\mu_2)}\right\rangle+\left\langle R\hat{k}_{(\lambda_1,\lambda_2)},\hat{k}_{(\lambda_1,\lambda_2)}\right\rangle \left\langle S\hat{k}_{(\mu_1,\mu_2)},\hat{k}_{(\mu_1,\mu_2)}\right\rangle\\
 &&+\left\langle P^2\hat{k}_{(\lambda_1,\lambda_2)},\hat{k}_{(\lambda_1,\lambda_2)}\right\rangle^{\frac{1}{2}}\left\langle R^2\hat{k}_{(\lambda_1,\lambda_2)},\hat{k}_{(\lambda_1,\lambda_2)}\right\rangle^{\frac{1}{2}}+\left|\left\langle T\hat{k}_{(\lambda_1,\lambda_2)},\hat{k}_{(\lambda_1,\lambda_2)}\right\rangle\right|\\
 &\leq& \Big(\left\langle P\hat{k}_{(\lambda_1,\lambda_2)},\hat{k}_{(\lambda_1,\lambda_2)}\right\rangle^2+\left\langle R\hat{k}_{(\lambda_1,\lambda_2)},\hat{k}_{(\lambda_1,\lambda_2)}\right\rangle^2\Big)^{\frac{1}{2}}\Big(\left\langle Q\hat{k}_{(\mu_1,\mu_2)},\hat{k}_{(\mu_1,\mu_2)}\right\rangle^2\\
 &&+\left\langle S\hat{k}_{(\mu_1,\mu_2)},\hat{k}_{(\mu_1,\mu_2)}\right\rangle^2\Big)^{\frac{1}{2}}
 +\frac{1}{2}\left\langle (P^2+R^2)\hat{k}_{(\lambda_1,\lambda_2)},\hat{k}_{(\lambda_1,\lambda_2)}\right\rangle+\left|\left\langle T\hat{k}_{(\lambda_1,\lambda_2)},\hat{k}_{(\lambda_1,\lambda_2)}\right\rangle\right|\\
 &=&\left|\left\langle P\hat{k}_{(\lambda_1,\lambda_2)},\hat{k}_{(\lambda_1,\lambda_2)}\right\rangle+i\left\langle R\hat{k}_{(\lambda_1,\lambda_2)},\hat{k}_{(\lambda_1,\lambda_2)}\right\rangle\right| \left|\left\langle Q\hat{k}_{(\mu_1,\mu_2)},\hat{k}_{(\mu_1,\mu_2)}\right\rangle+i\left\langle S\hat{k}_{(\mu_1,\mu_2)},\hat{k}_{(\mu_1,\mu_2)}\right\rangle\right|\\
  &&+\frac{1}{2}\left\langle (P^2+R^2)\hat{k}_{(\lambda_1,\lambda_2)},\hat{k}_{(\lambda_1,\lambda_2)}\right\rangle+\left|\left\langle T\hat{k}_{(\lambda_1,\lambda_2)},\hat{k}_{(\lambda_1,\lambda_2)}\right\rangle\right|\\
  &=&\left|\left\langle(P+iR)\hat{k}_{(\lambda_1,\lambda_2)},\hat{k}_{(\lambda_1,\lambda_2)}\right\rangle\right|\left|\left\langle (Q+iS)\hat{k}_{(\mu_1,\mu_2)},\hat{k}_{(\mu_1,\mu_2)}\right\rangle\right|\\
  &&+\frac{1}{2}\left\langle (P^2+R^2)\hat{k}_{(\lambda_1,\lambda_2)},\hat{k}_{(\lambda_1,\lambda_2)}\right\rangle+\left|\left\langle T\hat{k}_{(\lambda_1,\lambda_2)},\hat{k}_{(\lambda_1,\lambda_2)}\right\rangle\right|.
\end{eqnarray*}
 Taking the supremum over all $(\lambda_1,\lambda_2),(\mu_1,\mu_2)\in\Omega^2,$ we get
  \begin{eqnarray}\label{inq7}
\left\|
	\begin{bmatrix}
		A & B\\
		C & D
	\end{bmatrix}
\right\|_{ber}^2\leq \textbf{ber}(P+iR)\textbf{ber}(Q+iS)+\frac{1}{2}\textbf{ber}(P^2+R^2)+\textbf{ber}(T).
  \end{eqnarray}
By using Lemma \ref{lm7} (i) and Corollary \ref{c28} (ii) in (\ref{inq7}), we get the desired inequality.
\end{proof}

\begin{remark}
Following \cite[Theorem 3.6]{BHLS_RMJ_2021}, for the case $\alpha=\frac{1}{2},$ we have 
\begin{eqnarray}\label{ee5}
	\textbf{ber}\begin{pmatrix}
		\begin{bmatrix}
			A & B\\
			C & D
	\end{bmatrix}\end{pmatrix}
&\leq& \frac{1}{2} \textbf{ber}(D)+\textbf{ber}(A)\nonumber\\
&&+\frac{1}{2}\sqrt{\frac14\textbf{ber}^2(D)+\|C\|^2}+\frac{1}{2}\sqrt{\frac14\textbf{ber}^2(D)+\|B\|^2}.
\end{eqnarray}
If we take $\mathcal{H}=\mathbb{C}^2,A=C^*=\begin{bmatrix}
	1 & 1\\
	0 & 0
\end{bmatrix}$ and $B=D^*=\begin{bmatrix}
0 & 0\\
1 & 1
\end{bmatrix}$ then by simple computation from \eqref{ee5} we have $\textbf{ber}\begin{pmatrix}
\begin{bmatrix}
	A & B\\
	C & D
	\end{bmatrix}\end{pmatrix}
	\leq 3,$ whereas from Theorem \ref{th7} we get $\textbf{ber}\begin{pmatrix}
		\begin{bmatrix}
			A & B\\
			C & D
	\end{bmatrix}\end{pmatrix}
	\leq 2.4.$ Hence for this example the bound obtained in Theorem \ref{th7} is better than the existing bound \eqref{ee5}.
\end{remark}

Next we need the following lemmas to prove our next result.

\begin{lemma}\cite{VK_MB_71}\label{lm9}
	For $i=1,2,\ldots,n$, let $a_i$ be a positive real number. Then 
	\[\left( \sum_{i=1}^na_i\right)^r \leq n^{r-1}\sum_{i=1}^na_i^r,\] for all $r\geq 1$.
\end{lemma}

\begin{lemma}\cite[Lemma 2.10]{Bak_CMJ_2018}\label{lm10}
	Let $A,B\in \mathbb{B}(\mathcal{H}).$ If $\textbf{ber}\begin{pmatrix}\begin{bmatrix}
		A&0\\
		0&0
	\end{bmatrix}\end{pmatrix} \leq \textbf{ber}\begin{pmatrix}\begin{bmatrix}
		B&0\\
		0&0
	\end{bmatrix}\end{pmatrix},$ then $\textbf{ber}(A) \leq \textbf{ber}(B).$
\end{lemma}

\begin{lemma}\label{lm11}
	Let $A_i\in \mathbb{B}(\mathcal{H})$ be positive operators, for all $i=1,2,\ldots,n.$ Then 
	\begin{eqnarray*}
		\left\|\sum_{i=1}^n A_i\right\|^r_{ber} \leq n^{r-1}\left\|\sum_{i=1}^nA^r_i\right\|_{ber},
	\end{eqnarray*}
	for all $r \geq 1.$
\end{lemma}

\begin{proof}
	Let $\hat{k}_{\lambda}$ be a normalized reproducing kernel of $\mathcal{H}.$ Then using Lemma \ref{lm1} and Lemma \ref{lm9}, we obtain
	\begin{eqnarray*}
		\left\langle \sum_{i=1}^nA_i\hat{k}_{\lambda},\hat{k}_{\lambda} \right\rangle^r \leq n^{r-1}\sum_{i=1}^n\langle A_i \hat{k}_{\lambda},\hat{k}_{\lambda} \rangle^r \leq n^{r-1}\sum_{i=1}^n\langle A^r_i \hat{k}_{\lambda},\hat{k}_{\lambda} \rangle \leq n^{r-1}\left\|\sum_{i=1}^nA^r_i\right\|_{ber}.
	\end{eqnarray*}
	So, taking the supremum over all $\lambda\in\Omega$, we get the desired result.
\end{proof}

Now, we are in a position to prove the following result.
\begin{theorem}\label{th9}
	Let $A_i, B_i, X_i\in \mathbb{B}(\mathcal{H})$ for all $i=1,2,\ldots,n.$ Then 
	\begin{eqnarray*}
		\textbf{ber}^r\left(\sum_{i=1}^n A^*_iX_iB_i\right) \leq \frac{n^{r-1}}{2^r}\left(\max_{1\leq i \leq n}\|X_i\|^r\right)\left\|\sum_{i=1}^n\left(A_i^*A_i+B_i^*B_i\right)^r\right\|_{ber},
	\end{eqnarray*}
	for all $r \geq 1.$
\end{theorem}

\begin{proof}
	Consider the operator matrices \\
	
	$A=\begin{bmatrix}
		A_1&0&\ldots &0\\
		A_2&0&\ldots &0\\
		\vdots&\vdots&\vdots\\
		A_n&0&\ldots &0
	\end{bmatrix},$
	$B=\begin{bmatrix}
		B_1&0&\ldots &0\\
		B_2&0&\ldots &0\\
		\vdots&\vdots&\vdots\\
		B_n&0&\ldots &0
	\end{bmatrix}$ and $X=\begin{bmatrix}
		X_1&0&\ldots &0\\
		0&X_2&\ldots &0\\
		\vdots&\vdots&\vdots\\
		0&0&\ldots &X_n
	\end{bmatrix}.$ \\
	
	Then it is easy to verify that $A^*XB=\begin{bmatrix}
		\sum_{i=1}^n A^*_iX_iB_i&0&\ldots &0\\
		0&0&\ldots &0\\
		\vdots&\vdots&\vdots\\
		0&0&\ldots &0
	\end{bmatrix}.$
	
	Therefore, we have 
	\begin{eqnarray*}
		&&\textbf{ber}\begin{pmatrix}\begin{bmatrix}
			\sum_{i=1}^n A^*_iX_iB_i&0\\
			0&0
		\end{bmatrix}\end{pmatrix}\\
		&&=\textbf{ber}(A^*XB)\\
		&& =\sup\{|\langle A^*XB \hat{k}_{(\lambda_1,\lambda_2,\ldots,\lambda_n)},\hat{k}_{(\lambda_1,\lambda_2,\ldots,\lambda_n)}\rangle| : (\lambda_1,\lambda_2,\ldots,\lambda_n) \in \Omega^n\}\\
		&& =\sup\{|\langle XB \hat{k}_{(\lambda_1,\lambda_2,\ldots,\lambda_n)},A\hat{k}_{(\lambda_1,\lambda_2,\ldots,\lambda_n)}\rangle| : (\lambda_1,\lambda_2,\ldots,\lambda_n) \in \Omega^n\}\\
		&& \leq \|X\|\sup\{\| B \hat{k}_{(\lambda_1,\lambda_2,\ldots,\lambda_n)}\|\|A\hat{k}_{(\lambda_1,\lambda_2,\ldots,\lambda_n)}\| : (\lambda_1,\lambda_2,\ldots,\lambda_n) \in \Omega^n\}\\
		&& \leq \frac{\|X\|}{2}\sup\{\| B \hat{k}_{(\lambda_1,\lambda_2,\ldots,\lambda_n)}\|^2+\|A\hat{k}_{(\lambda_1,\lambda_2,\ldots,\lambda_n)}\|^2 : (\lambda_1,\lambda_2,\ldots,\lambda_n) \in \Omega^n\}\\
		&& = \frac{\|X\|}{2}\sup\{\langle (A^*A+B^*B) \hat{k}_{(\lambda_1,\lambda_2,\ldots,\lambda_n)},\hat{k}_{(\lambda_1,\lambda_2,\ldots,\lambda_n)}\rangle: (\lambda_1,\lambda_2,\ldots,\lambda_n) \in \Omega^n\}\\
		&&=\frac{1}{2}\left(\max_{1\leq i \leq n}\|X_i\|\right)\textbf{ber}(A^*A+B^*B)\\
		&&=\frac{1}{2}\left(\max_{1\leq i \leq n}\|X_i\|\right)\textbf{ber}\begin{pmatrix}\begin{bmatrix}
			\sum_{i=1}^n (A^*_iA_i+B^*_iB_i)&0\\
			0&0
		\end{bmatrix}\end{pmatrix}.
	\end{eqnarray*}
	
	By using Lemma \ref{lm10}, we have
	
	\begin{eqnarray*}
		\textbf{ber}\left(\sum_{i=1}^n A^*_iX_iB_i\right) \leq \frac{1}{2}\left(\max_{1\leq i \leq n}\|X_i\|\right)\textbf{ber}\left(\sum_{i=1}^n (A^*_iA_i+B^*_iB_i)\right).
	\end{eqnarray*}

	Now, by using Lemma \ref{lm11}, we obtain
	\begin{eqnarray*}
		\textbf{ber}^r\left(\sum_{i=1}^n A^*_iX_iB_i\right) &&\leq \frac{1}{2^r}\left(\max_{1\leq i \leq n}\|X_i\|^r\right)\left\|\sum_{i=1}^n (A^*_iA_i+B^*_iB_i)\right\|^r_{ber}\\
		&& \leq \frac{n^{r-1}}{2^r}\left(\max_{1\leq i \leq n}\|X_i\|^r\right)\left\|\sum_{i=1}^n\left(A_i^*A_i+B_i^*B_i\right)^r\right\|_{ber},
	\end{eqnarray*}
	as desired.
\end{proof}
The following corollaries are immediate from Theorem \ref{th9}.
\begin{cor}\label{cot9}
	Let $A_i,B_i \in \mathbb{B}(\mathcal{H}),$ for all $i=1,2,\ldots,n.$ Then 
\begin{eqnarray*}
	&&(i)~~\textbf{ber}^r\left(\sum_{i=1}^n A_iB_i\right) \leq \frac{n^{r-1}}{2^r}\left\|\sum_{i=1}^n\left(A_iA_i^*+B_i^*B_i\right)^r\right\|_{ber} \mbox{for all $r \geq 1,$}\\
	&&(ii)~~ \textbf{ber}^r\left(\sum_{i=1}^n A_i\right) \leq \frac{n^{r-1}}{2^r}\min \left\{\left\|\sum_{i=1}^n\left(A_iA_i^*+I\right)^r\right\|_{ber},\left\|\sum_{i=1}^n\left(A^*_iA_i+I\right)^r\right\|_{ber}\right\}\\
	&&\,\,\,\,\,\,\,\,\,\,~~\mbox{for all $r \geq 1,$}\\
	&&(iii)~~ \textbf{ber}\left(\sum_{i=1}^n A_iB_i\right) \leq \frac{1}{2}\left\|\sum_{i=1}^n\left(A_iA_i^*+B_i^*B_i\right)\right\|_{ber},\\
	&&(iv)~~ \textbf{ber}\left(A_{1}B_{1}+A_{2}B_{2}\right) \leq \frac{1}{2}\left\|A_{1}A_{1}^*+B_{1}^*B_{1}+A_{2}A_{2}^*+B_{2}^*B_{2}\right\|_{ber}.
\end{eqnarray*}
\end{cor}

\begin{cor}\label{cot10}
	Let $A,B,X \in \mathbb{B}(\mathcal{H}).$ Then 
	\begin{eqnarray*}
		&&(i)~~\textbf{ber}^r\left(A^*XB\right) \leq \frac{\|X\|^r}{2^r}\|(A^*A+B^*B)^r\|_{ber}\,\,\,~~ \mbox{for all $r \geq 1,$}\\
		&&(ii)~~ \textbf{ber}^r\left(A^*B\right) \leq \frac{\|(A^*A+B^*B)^r\|_{ber}}{2^r} \,\,\,~~\mbox{for all $r \geq 1.$}
	\end{eqnarray*}
\end{cor}

\begin{cor}\label{cot11}
	Let $A,B,X \in \mathbb{B}(\mathcal{H})$ where $A$ and $B$ are positive. Then for all $r \geq 1$ and $\alpha \in [0,1]$
	\begin{eqnarray*}
		&&(i)~~\textbf{ber}^r\left(A^{\alpha}XB^{1-\alpha}\right) \leq \frac{\|X\|^r}{2^r}\|(A^{2\alpha}+B^{2(1-\alpha)})^r\|_{ber},\\
		&&(ii)~~ \textbf{ber}^r\left(A^{\alpha}B^{1-\alpha}\right) \leq \frac{\|(A^{2\alpha}+B^{2(1-\alpha)})^r\|_{ber}}{2^r}. 
	\end{eqnarray*}
	In particular, if $AB=BA$ then 
	\begin{eqnarray*}
		\left\|\sqrt{AB}\right\|_{ber}^r \leq \left\|\left(\frac{A+B}{2}\right)^r\right\|_{ber}.
	\end{eqnarray*}
\end{cor}

 \begin{remark}
 	(i) Following \cite[Corollary 2.19 (ii)]{BSP_ACTA_2023}, for the case $f(t)= g(t)= t^{1/2}$ and $r=n=2,$ we get 
 	\begin{eqnarray}\label{ee1}
 		\textbf{ber}^2(A_{1}+A_{2})\leq \sqrt{2} \textbf{ber} \left(|A_{1}|^2+|A_{2}|^2+ i(|A^*_{1}|^2+|A^*_{2}|^2)\right).
 	\end{eqnarray}
 	If we consider $\mathcal{H}=\mathbb{C}^2,$ $A_{1}=\begin{bmatrix}
 		1&0\\
 		0&0
 	\end{bmatrix}$ and $A_{2}=\begin{bmatrix}
 	0&1\\
 	0&0
 	\end{bmatrix}$ then from \eqref{ee1} we have $ \textbf{ber}^2(A_{1}+A_{2})\leq \sqrt{10}.$ Again Corollary \ref{cot9} (ii) gives $ \textbf{ber}^2(A_{1}+A_{2})\leq 2.5,$ for the case $n=r=2.$ Therefore, for this example the bound obtained in Corollary \ref{cot9} (i) is sharper than the existing bound given by \eqref{ee1}.\\
 	(ii) In \cite[Corollary 2.7]{sen_FILOMAT}, Sen et al. obtained the following result, namely,
 	\begin{eqnarray}\label{ee2}
 		\textbf{ber}^r (A^*B)\leq \frac{1}{2} \sqrt{\textbf{ber}^2(|A|^{2r}+|B|^{2r})-c^2(|A|^{2r}-|B|^{2r})}
 	\end{eqnarray}
 	where $c(A)=\inf_{\lambda\in\Omega}|\widetilde{A}(\lambda)|.$
 		Considering  $A=\begin{bmatrix}
 		1&1\\
 		0&0
 	\end{bmatrix}$ and $B=\begin{bmatrix}
 		0&1\\
 		0&0
 	\end{bmatrix}$ then from Corollary \ref{cot10} (ii) (for $r=2$) we get  $ \textbf{ber}^2(A^*B)\leq 1.25,$ whereas for $r=2$ the inequality \eqref{ee2} gives $ \textbf{ber}^2(A^*B)\leq 1.41.$ Hence, for this example Corollary \ref{cot10} (ii) is better than the existing result given by \eqref{ee2}.\\
 	(iii) For positive operators $A$ and $B,$ the following upper bound was obtained in \cite[Corollary 2.10]{sen_FILOMAT}, namely,
 	\begin{eqnarray}\label{ee3}
 		\textbf{ber}^4(A^{1/2}B^{1/2}) \leq \frac14\left(\|A^2+B^2\|_{ber}-c^2(A^2-B^2)\right).
 	\end{eqnarray}
 If we take $A=\begin{bmatrix}
 	1/4&0\\
 	0&1/2
 \end{bmatrix}$ and $B=\begin{bmatrix}
 	1/4&0\\
 	0&0
 \end{bmatrix}$ then Corollary \ref{cot11} (ii) (for $r=4, \alpha=\frac12$) gives $ \textbf{ber}^2(A^{1/2}B^{1/2})\leq {1/2^8},$ whereas the inequality \eqref{ee3} gives $ \textbf{ber}^2(A^{1/2}B^{1/2})\leq 1/{2^6}.$ Therefore, for this example we can conclude that the bound of Corollary \ref{cot11} (ii) is sharper than the existing bound \eqref{ee3}.
 \end{remark}

To prove our next theorem, we use the following lemma which is a generalization of Lemma \ref{lm6}.
\begin{lemma}\label{lm13} 
	If $x_{1}, x_{2}, \dots ,x_{n},e \in \mathcal{H}$ with $\|e\|=1 \text{ and } \alpha \in\mathbb{C}\setminus \{0\},$  then
	\begin{eqnarray*}
		\left|\prod_{i=1}^{n}\langle x_{i},e\rangle\right|\leq \frac{1}{|\alpha|}\left(\left| \langle x_{1},x_{2}\rangle \prod_{i=3}^{n}\langle x_{i},e\rangle\right|+\max\{1,|\alpha-1|\}\prod_{i=1}^{n}\|x_{i}\|\right).
	\end{eqnarray*}
\end{lemma}
\begin{proof}
	Substituting   $x_{2}$ by $\prod_{i=3}^{n}\langle x_{i},e\rangle x_{2}$ and applying $\big|\prod_{i=3}^{n}\langle x_{i},e\rangle\big|\leq \prod_{i=3}^{n}\|x_{i}\|$ in Lemma \ref{lm6}, we get the desired result.
\end{proof}

\begin{theorem}\label{th10}
	Let	$A\in \mathbb{B}(\mathcal{H})$ . Then
	\begin{eqnarray*}
		\textbf{ber}^3(A)&\leq& \inf_{\alpha,\beta\in\mathbb{C}\backslash\{0\}}\Bigg\{\frac{1}{|\alpha||\beta|} \textbf{ber}(A^3)
	 +\Bigg(\frac{\max\left\{1, |\beta-1|\right\}}{|\alpha||\beta|}\left\|{A^*}^2 A^2\right\|^{\frac{1}{2}}_{ber}\\
	 &&+\frac{\max\left\{1, |\alpha-1|\right\}}{2|\alpha|}\left\||A|^2+ |A^*|^2\right\|_{ber}\Bigg)\|AA^*\|^{\frac{1}{2}}_{ber}\Bigg\}.
	\end{eqnarray*}
	\end{theorem}
	\begin{proof}
		Let $\hat{k}_{\lambda}$ be a normalized reproducing kernel of $\mathcal{H}$. Then using Lemma \ref{lm13}, we get
		\begin{eqnarray*}
			|\langle A\hat{k}_{\lambda},\hat{k}_{\lambda}\rangle|^3
			&=& |\langle A\hat{k}_{\lambda},\hat{k}_{\lambda}\rangle \langle A^*\hat{k}_{\lambda},\hat{k}_{\lambda}\rangle\langle A^*\hat{k}_{\lambda},\hat{k}_{\lambda}\rangle|\\
			&\leq& \frac{1}{|\alpha|} \left(|\langle A\hat{k}_{\lambda},A^*\hat{k}_{\lambda}\rangle\langle A^*\hat{k}_{\lambda},\hat{k}_{\lambda}\rangle|+\max\left\{1, |\alpha-1|\right\}\|A\hat{k}_{\lambda}\|\|A^*\hat{k}_{\lambda}\|^2\right)\\
			&\leq& \frac{1}{|\alpha|}\Bigg(\frac{1}{|\beta|}\Big(|\langle A^2\hat{k}_{\lambda},A^*\hat{k}_{\lambda}\rangle| +\max\left\{1, |\beta-1|\right\}\|A^2\hat{k}_{\lambda}\|\|A^*\hat{k}_{\lambda}\|\Big)\\
			&&+\max\left\{1, |\alpha-1|\right\}\|A\hat{k}_{\lambda}\|\|A^*\hat{k}_{\lambda}\|^2\Bigg)\\
			&=& \frac{1}{|\alpha|}\Bigg(\frac{1}{|\beta|}\Big(|\langle A^3\hat{k}_{\lambda},\hat{k}_{\lambda}\rangle| +\max\left\{1, |\beta-1|\right\}\|A^2\hat{k}_{\lambda}\|\|A^*\hat{k}_{\lambda}\|\Big)\\
			&&+\max\left\{1, |\alpha-1|\right\}\|A\hat{k}_{\lambda}\|\|A^*\hat{k}_{\lambda}\|^2\Bigg)\\
			&\leq& \frac{1}{|\alpha||\beta|}|\langle A^3\hat{k}_{\lambda},\hat{k}_{\lambda}\rangle| +\Bigg(\frac{\max\left\{1, |\beta-1|\right\}}{|\alpha||\beta|}\|A^2\hat{k}_{\lambda}\|\\
			&&+\frac{\max\left\{1, |\alpha-1|\right\}}{2|\alpha|}\big(\|A\hat{k}_{\lambda}\|^2+\|A^*\hat{k}_{\lambda}\|^2\big)\Bigg)\|A^*\hat{k}_{\lambda}\|\\
			&=&\frac{1}{|\alpha||\beta|}|\langle A^3\hat{k}_{\lambda},\hat{k}_{\lambda}\rangle| +\Bigg(\frac{\max\left\{1, |\beta-1|\right\}}{|\alpha||\beta|}\langle {A^*}^2 A^2\hat{k}_{\lambda},\hat{k}_{\lambda}\rangle^{\frac{1}{2}}\\
			&&+\frac{\max\left\{1, |\alpha-1|\right\}}{2|\alpha|}\langle (|A|^2+|A^*|^2)\hat{k}_{\lambda},\hat{k}_{\lambda}\rangle\Bigg)\langle AA^*\hat{k}_{\lambda},\hat{k}_{\lambda}\rangle^{\frac{1}{2}}\\
			&\leq & \frac{1}{|\alpha||\beta|} \textbf{ber}(A^3)
			+\Bigg(\frac{\max\left\{1, |\beta-1|\right\}}{|\alpha||\beta|}\left\|{A^*}^2 A^2\right\|^{\frac{1}{2}}_{ber}\\
			&&+\frac{\max\left\{1, |\alpha-1|\right\}}{2|\alpha|}\left\||A|^2+ |A^*|^2\right\|_{ber}\Bigg)\|AA^*\|^{\frac{1}{2}}_{ber}.
		\end{eqnarray*}
		Taking the supremum over all $\lambda\in\Omega$, we obtain
		\begin{eqnarray}\label{th10eq1}
			\textbf{ber}^3(A)&\leq& \frac{1}{|\alpha||\beta|} \textbf{ber}(A^3)
			+\Bigg(\frac{\max\left\{1, |\beta-1|\right\}}{|\alpha||\beta|}\left\|{A^*}^2 A^2\right\|^{\frac{1}{2}}_{ber}\nonumber\\
			&&+\frac{\max\left\{1, |\alpha-1|\right\}}{2|\alpha|}\left\||A|^2+ |A^*|^2\right\|_{ber}\Bigg)\|AA^*\|^{\frac{1}{2}}_{ber}.
		\end{eqnarray}
	The required result follows from \eqref{th10eq1} by taking infimum over all $\alpha,\beta\in\mathbb{C}\backslash\{0\}.$
	\end{proof}

Considering  $\alpha=\beta=2$ in Theorem \ref{th10}, we get the following corollary.

\begin{cor}\label{th10cor1}
		Let	$A\in \mathbb{B}(\mathcal{H})$. Then  
		\begin{eqnarray*}
			\textbf{ber}^3(A) \leq \frac{1}{4} \textbf{ber}(A^3)
			+\frac{1}{4}\left(\left\|{A^*}^2 A^2\right\|^{\frac{1}{2}}_{ber}+\left\||A|^2+ |A^*|^2\right\|_{ber}\right)\|AA^*\|^{\frac{1}{2}}_{ber}.
		\end{eqnarray*}
\end{cor}
\begin{remark}
	In \cite[Corollary 3.5(i)]{TRD_FILOMAT_19}, Taghavi et al. obtained the following bound of Berezin number, namely,
	\begin{eqnarray}\label{ee4}
		\textbf{ber}^r (A)\leq \frac{1}{2} 	\textbf{ber}(|A|^r+|A^*|^r) ~~\mbox{for all $r\geq 1$ .} 
	\end{eqnarray}
	If we consider $\mathcal{H}=\mathbb{C}^3$ and $A=\begin{bmatrix}
		0&1&0\\
		0&0&1\\
		0&0&0
	\end{bmatrix}$ then by simple computation from Corollary \ref{th10cor1} we get $\textbf{ber}^3(A)\leq 0.75,$ whereas for $r=3$ the inequality \eqref{ee4} gives $\textbf{ber}^3(A)\leq 1.$ Therefore, for this example the bound obtained in Corollary \ref{th10cor1} is better than the existing bound 
\eqref{ee4}.
\end{remark}

Proceeding similarly as Theorem \ref{th10} and by using Lemma \ref{lm8} and Lemma \ref{lm13}, we get the following upper bounds of Berezin number.

	\begin{theorem}\label{th11}
			Let	$A\in \mathbb{B}(\mathcal{H})$ . Then  
			\begin{eqnarray*}
			(i)~~\textbf{ber}^3(A)&\leq& \inf_{\alpha,\beta\in\mathbb{C}\backslash\{0\}}\Bigg\{\frac{1}{|\alpha||\beta|} \textbf{ber}(A^*|A^*||A|)
			+\Bigg(\frac{\max\left\{1, |\beta-1|\right\}}{|\alpha||\beta|}\left\||A||A^*|^2|A|\right\|^{\frac{1}{2}}_{ber}\\
			&&+\frac{\max\left\{1, |\alpha-1|\right\}}{2|\alpha|}\left\||A|^2+ |A^*|^2\right\|_{ber}\Bigg)\|A^*A\|^{\frac{1}{2}}_{ber}\Bigg\}\\
			& \leq & \frac14 \textbf{ber}(A^*|A^*||A|)+\frac14\left(\left\||A||A^*|^2|A|\right\|^{\frac{1}{2}}_{ber}+\left\||A|^2+ |A^*|^2\right\|_{ber}\right)\|A^*A\|^{\frac{1}{2}}_{ber},\\
			(ii)~\textbf{ber}^3(A)&\leq& \inf_{\alpha,\beta\in\mathbb{C}\backslash\{0\}}\Bigg\{\frac{1}{|\alpha||\beta|} \textbf{ber}(A|A||A^*|)
			+\Bigg(\frac{\max\left\{1, |\beta-1|\right\}}{|\alpha||\beta|}\left\||A^*||A|^2|A^*|\right\|^{\frac{1}{2}}_{ber}\\
			&&+\frac{\max\left\{1, |\alpha-1|\right\}}{2|\alpha|}\left\||A|^2+ |A^*|^2\right\|_{ber}\Bigg)\|AA^*\|^{\frac{1}{2}}_{ber}\Bigg\}\\
				& \leq & \frac14 \textbf{ber}(A|A||A^*|)+\frac14\left(\left\||A^*||A|^2|A^*|\right\|^{\frac{1}{2}}_{ber}+\left\||A|^2+ |A^*|^2\right\|_{ber}\right)\|AA^*\|^{\frac{1}{2}}_{ber}.\\
			(iii)~\textbf{ber}^4(A)&\leq&\frac{1}{4}\left(\textbf{ber}( |A^*||A|)+\frac{1}{2}\||A|^2+|A^*|^2\|_{ber}\right)
				\left(\textbf{ber}( {A}^2)+\frac{1}{2}\||A|^2+|A^*|^2\|_{ber}\right).\\
			(iv)~\textbf{ber}^4(A)&\leq&\frac{1}{4}\left(\textbf{ber}( A|A|)+\frac{1}{2}\||A|^2+|A^*|^2\|_{ber}\right)
			\left(\textbf{ber}( A^* |A^*|)+\frac{1}{2}\||A|^2+|A^*|^2\|_{ber}\right). 	
			\end{eqnarray*}                                                            
	\end{theorem}

Finally, we obtain the following general power inequality of Berezin number.

\begin{theorem}\label{T20}
		Let	$A\in \mathbb{B}(\mathcal{H})$ . Then 
		\begin{eqnarray*}
			\textbf{ber}^n(A) \leq \frac{1}{2^{n-1}}\textbf{ber}(A^n)+\sum_{i=1}^{n-1}\frac{1}{2^i}\|A^{*i}A^i\|^{1/2}_{ber}\|AA^*\|^{\frac{n-i}{2}}_{ber},
		\end{eqnarray*}
	for every positive integers $n\geq 2.$
\end{theorem}

\begin{proof}
	For every positive integer $n \geq 2$ and for all $x \in \mathcal{H}$ with $\|x\|=1,$ the following inequality (\cite[Theorem 3.1]{B_arxiv_23}) holds 
	\begin{eqnarray}\label{T20eq1}
		|\langle Ax,x \rangle|^n \leq \frac{1}{2^{n-1}}|\langle A^nx,x \rangle|+\sum_{i=1}^{n-1}\frac{1}{2^i}\|A^ix\|\|A^*x\|^{n-i}.
	\end{eqnarray}
	Let $\hat{k}_{\lambda}$ be a normalized reproducing kernel of $\mathcal{H}$. Now, putting $x=\hat{k}_{\lambda}$ in \eqref{T20eq1}, we get
	\begin{eqnarray*}
		|\langle A\hat{k}_{\lambda},\hat{k}_{\lambda}\rangle|^n
	 &\leq& \frac{1}{2^{n-1}}|\langle A^n\hat{k}_{\lambda},\hat{k}_{\lambda} \rangle|+\sum_{i=1}^{n-1}\frac{1}{2^i}\|A^i\hat{k}_{\lambda}\|\|A^*\hat{k}_{\lambda}\|^{n-i}\\
		&=& \frac{1}{2^{n-1}}|\langle A^n\hat{k}_{\lambda},\hat{k}_{\lambda} \rangle|+
		\sum_{i=1}^{n-1}\frac{1}{2^i}\langle A^{*i}A^i \hat{k}_{\lambda}, \hat{k}_{\lambda} \rangle^{1/2}\langle AA^* \hat{k}_{\lambda}, \hat{k}_{\lambda} \rangle^{\frac{n-i}{2}}\\
		&\leq & \frac{1}{2^{n-1}}\textbf{ber}(A^n)+\sum_{i=1}^{n-1}\frac{1}{2^i}\|A^{*i}A^i\|^{1/2}_{ber}\|AA^*\|^{\frac{n-i}{2}}_{ber}.
	\end{eqnarray*}
Taking the supremum over all $\lambda\in\Omega$ we get the desired result.
\end{proof}

\noindent {\textbf{Data availability statement.}}\\
Data sharing is not applicable to this article as no new data were created or analyzed in this study.\\

\noindent {\textbf{Disclosure statement.}}\\
There are no relevant financial or non-financial competing interests to report.

	\bibliographystyle{amsplain}

\end{document}